\documentclass[11pt]{amsart}
\usepackage[english]{babel}
\usepackage{amsfonts,latexsym,amsthm,amssymb,graphicx}
\usepackage{mathabx}
\usepackage[all]{xy}
\usepackage[usenames]{color}
\usepackage{tikz}
\usetikzlibrary{shapes}
\usepackage{xypic}
\usetikzlibrary{decorations.pathreplacing}
\usepackage{amsmath}
\usepackage{graphicx}
\usepackage[enableskew]{youngtab}
\usepackage[utf8]{inputenc}
\usepackage[english]{babel}
\usepackage{tikz}
\usetikzlibrary{arrows}
\usepackage{float}
\usepackage{pgf}
\usepackage{subcaption}
\usepackage{pgfplots}
\usepackage{amscd}
\usepackage{tikz}
\usepackage{blkarray}
\usepgflibrary{decorations.shapes}
\usetikzlibrary{shapes.geometric, decorations.shapes,decorations.markings, shapes, fit, arrows, positioning, trees, mindmap, calc}
\usepackage[colorlinks=true, linkcolor=blue, citecolor=blue, urlcolor=blue]{hyperref}

\newcommand{\Z}{{\mathbb Z}}

\DeclareMathOperator{\QH}{QH}

\newcommand{\OG}{\mathrm{OG}}

\newcommand{\GCD}{\mathrm{GCD}}

\CompileMatrices

\newtheorem{thm}{Theorem}[section]
\newtheorem{lemma}[thm]{Lemma}
\newtheorem{cor}[thm]{Corollary}
\newtheorem{prop}[thm]{Proposition}

{\theoremstyle{defn} \newtheorem{defn}[thm]{Definition}}

{\theoremstyle{remark} 
\newtheorem{example}[thm]{Example}}

\newenvironment{tightcenter}{%
  \setlength\topsep{0pt}
  \setlength\parskip{0pt}
  \begin{center}
}{%
  \end{center}
}

\begin{document}

\title{On the spectral properties of the quantum cohomology of odd quadrics}

\author{Ryan M. Shifler} 
\address{ Department of Mathematical Sciences, Henson Science Hall, Salisbury University, Salisbury MD 21801 USA}\email{rmshifler@salisbury.edu}

\author{Stephanie Warman} 
\address{ Department of Mathematical Sciences, Henson Science Hall, Salisbury University, Salisbury MD 21801 USA}\email{swarman2@gulls.salisbury.edu}


\begin{abstract}
Let $H^\bullet(\OG)$ be the quantum cohomology (specialized at $q=1$) of the $2n-1$ dimensional quadric $\OG$. We will calculate the characteristic polynomial of the linear operators induced by quantum multiplication in $H^\bullet(\OG)$ and the Frobenius-Perron dimension. We also check that Galkin's lower bound conjecture holds for $\OG$.
 \end{abstract}
\maketitle

\section{Introduction}\label{s:intro}  
Let $\OG:=\OG(1,2n+1)$ be the $2n-1$ dimensional quadric.  This is the parameterization of isotropic 1 dimensional subspaces of $\mathbb{C}^{2n+1}$, isotropic with respect to a symmetric non-degenerate symmetric form. The quantum cohomology ring $(\mathrm{QH}^*(\OG), \star)$ is a graded algebra over $\Z[q]$, where $q$ is the quantum parameter and $\deg q=2n-1$. The ring $(\mathrm{QH}^*(\OG), \star)$ has Schubert classes given by $\tau_p$, with $0 \leq p \leq 2n-1$. The ring multiplication is given by \[ \tau_p \star \tau_i = \sum_{j, d_0 \geq 0} c_{p, i}^{j,d_0}q^{d_0} \tau_j \] where $c_{p, i}^{j,d_0}$ is the Gromov-Witten invariant that enumerates degree $d_0$ rational curves intersecting general translate of $\tau_{p}, \tau_{i},$ and the Poincar\'e dual of $\tau_{j}$ which is $\tau_{2n-1-j}$.

Consider the specialization $H^\bullet(\OG):= \mathrm{QH}^*(\OG)_{|q=1}$ at $q=1$. The quantum multiplication by the Schubert class $\tau_p$, with $0 \leq p \leq 2n-1$, induces an endomorphism $A(\tau_p)$ of the finite-dimensional vector space $H^\bullet(\OG)$:
 \[ y\in H^\bullet(\OG) \mapsto A( \tau_p)(y):= (\tau_p\star y)|_{q=1} \/. \]
 
 Denote by $\delta_p:=\max\{|\delta|:\delta \mbox{ is an eigenvalue of } A(\tau_p)\}$ the real largest eigenvalue of $A(\tau_p)$. Define the Frobenius-Perron Dimension to be \[ \text{FPdim}:H^\bullet(\OG) \rightarrow \mathbb{C} \mbox{ defined by } \text{FPdim}(\tau_p)=\delta_p.\] The first study of Frobenius-Perron dimension on quantum cohomology specialized at $q=1$ is for the Grassmannian in \cite{Riet} with subsequent study in \cite{LSYZ}.
 
 To simplify notation we let $d:=\GCD(p,2n-1)$ for the remainder of the article. We will prove the next two results.
 
  \begin{thm} \label{main:thm}
The characteristic polynomial of $A(\tau_p)$ is

\[ \begin{cases} 
      \lambda\left(\lambda^{\frac{2n-1}{d}}-2^{\frac{2p}{d}}\right)^{d} & 1 \leq p <n \\
     \lambda\left(\lambda^{\frac{2n-1}{d}}- 2^{\frac{2p-(2n-1)}{d}}\right)^{d} & n \leq p <2n-1 \\
      (\lambda -1)^{2n-1}(\lambda+1) & p=2n-1
   \end{cases}.
\]
 In particular, the eigenvalues of $A(\tau_p)$ are simple if and only if the integers $p$ and $2n-1$ are relatively prime.
 \end{thm}
The next corollary follows immediately from Theorem \ref{main:thm}.
 \begin{cor}
 We have that $\text{FPdim}(\tau_p)=
 \begin{cases}
 2^{\frac{2p}{2n-1}} & 1 \le p < n \\
 2^{\frac{2p}{2n-1}-1} & n \le p < 2n-1\\
 1 & p = 2n-1.
 \end{cases}
 $
 \end{cor}
 
 The final result will be a proof of a lower bound on the real largest eigenvalue of the first Chern class. We will state this precisely next. Let $K:=K_{\OG}$ be the canonical bundle of $\OG$ and let $c_1:=c_1(-K) \in H^2(\OG)$ be the anticanonical class. Galkin's lower bound conjecture for the $\OG$ case\footnote{Galkin's lower bound conjecture is for any fano variety.} states that (see \cite{Galkin}):
 \[ \text{FPdim}(A(c_1)) \geq \dim_{\mathbb{C}}\OG+1.\] Here, $\dim_{\mathbb{C}}\OG=2n-1$ and $c_1= (2n-1)\tau_1$.
 
 \begin{thm}
 Galkin's lower bound conjecture holds for $\OG$.
 \end{thm}
 
 The result is proved for other varieties in \cite{ESSSW,LSYZ,CH,Ke}.
 
 {\em Acknowledgements.} Shifler is partially supported by the Building Research Excellence (BRE) Program at Salisbury University. The authors would like to thank the reviewer for carefully reading the manuscript and providing useful comments and recommendations.
 
 \section{Preliminaries}
 \label{prelim}
We will rigorously define $\OG$. The \textit{reverse dot product} of two vectors $\left( a_1, a_2 \ldots a_n \right)$ and $\left( b_1, b_2, \ldots b_n \right)$ is 
\begin{eqnarray}
\left( \left(  a_1, a_2 \ldots a_n \right),  \left( b_1, b_2, \ldots b_n \right) \right) =\sum_{i=1}^{n} a_ib_{n+1-i}. \label{revdot}
\end{eqnarray}
We are ready to define $\OG$. 

\begin{defn}\label{defn OG(1,2n+1)}
The odd dimensional quadric $\OG$ is  \[\OG = \left\{V \subset \mathbb{C}^{2n+1} : \dim V=1, \left( a, b \right) =0 \;\forall \;a,b \in V  \right\}.\]
Equivalently, $\OG$ is the loci of $v_{n+1}^2 + 2 \sum_{i=1}^n v_iv_{2n+2-i}$.
\end{defn}

\subsection{Schubert cells and varieties}
The Schubert cells and Schubert varieties are denoted by $X(p)^\circ$ and $X(p)$, respectively, with $0 \leq p \leq 2n-1$. We will first state what the Schubert cells are and follow with a description of the Schubert varieties. 

The Schubert cells are given as follows.
\begin{enumerate}
\item For $0 \leq p \leq n$, we have that \[X(p)^\circ=\left\{\mbox{span}\{(\overbrace{0,\dots,0}^{p},1,\overbrace{*,\dots,*}^{2n-1-2p},\bullet,\overbrace{*,\dots,*}^{p} )\}:* \mbox{ is free and $\bullet$ is given by (\ref{revdot})} \right\}.\]

\item For $n+1 \leq p \leq 2n-1$, we have that
\[ X(p)^\circ=\left\{\mbox{span}\{(\overbrace{0,\dots,0}^{p+1},1,\overbrace{*,\dots,*}^{2n-1-p} )\}:* \mbox{ is free } \right\}.\]
\end{enumerate}

The Schubert variety $X(p)$ is the closure of the Schubert cell $X(p)^\circ$. Precisely, the Schubert variety $X(p)$ is given by
 \[ X(p)=\bigcup_{i \geq p} X(i)^\circ.\] Furthermore, $\dim X(p)=\dim X(p)^\circ=(2n-1)-p.$
 
\subsection{Quantum cohomology}
 We will now introduce the quantum cohomology ring of $\OG$. The (small) quantum cohomology of $\OG$, denoted by $\mbox{QH}^*(\OG)$, has a graded $\mathbb{Z}[q]$-basis consisting of Schubert classes $\tau_p:=[X(p)]$ where $0 \leq p \leq 2n-1$ and $\deg q=2n-1$. The ring multiplication is given by \[ \tau_p \star \tau_i = \sum_{j, d_0 \geq 0} c_{p, i}^{j,d_0}q^{d_0} \tau_j \] where $c_{p, i}^{j,d_0}$ is the Gromov-Witten invariant that enumerates degree $d_0$ rational curves intersecting general translate of $\tau_{p}, \tau_{i},$ and the Poincar\'e dual of $\tau_{j}$ which is $\tau_{2n-1-j}$. We refer the reader to  \cite{FP,kontsevich.manin:GW:qc:enumgeom} for additional details regarding quantum cohomology.
 
The quantum Chevalley formula for $\OG$ is given as follows (see \cite{BKT,BKT2}).

\[ \tau_1 \star \tau_p=\begin{cases} 
      \tau_{p+1} & 1 \leq p \leq n-1 \\
      2\tau_{n+1} & p=n \\
      \tau_{p+1} & n+1 \leq p \leq 2n-3\\
      \tau_{2n-1}+q\tau_0 & p=2n-2\\
      q\tau_1 & p=2n-1.
   \end{cases}
\]

Consider the specialization $H^\bullet(\OG):= \mathrm{QH}^*(\OG)_{|q=1}$ at $q=1$. The quantum multiplication by $\tau_p$  induces an endomorphism $A(\tau_p)$ of the finite-dimensional vector space $H^\bullet(\OG)$:
 \[ y\in H^\bullet(\OG \mapsto A( \tau_p)(y):= (A(\tau_p)\star y)|_{q=1} \/. \]

\begin{example}\label{example QH(OG(1,2n+1))}
The matrix $A(\tau_1)$ for $\OG(1,3)$ is computed below.
\[ A(\tau_1) = \begin{blockarray}{ccccc}
& \tau_0 & \tau_1 &\tau_2 &\tau_3 \\
\begin{block}{c(cccc)}
\tau_0   &0       &0        &1     &0 \\
\tau_1   &1       &0        &0     & 1\\
\tau_2   &0       &2        &0     &0 \\
\tau_3   &0       &0        & 1   &0 \\
\end{block}
\end{blockarray}
.\]
\end{example}
The next proposition follows from immediately from the quantum Chevalley formula.
 \begin{prop} For $0 \leq p \leq 2n-1$, we have that
 \[ A(\tau_p)=\begin{cases} 
      A(\tau_1)^p & 1\leq p \leq n-1 \\
      \frac{1}{2}A(\tau_1)^p & n\leq p\leq 2n-2\\
     \frac{1}{2}A(\tau_1)^{2n-1}-A(\tau_0)& p=2n-1.
   \end{cases}
\]
Recall that $A(\tau_0)$ is the identity matrix.
\end{prop}
 \section{Spectral properties of $A( \tau_1)$}
 
 We will study the Spectral properties of $A( \tau_1)$ in this section.
 \begin{prop}\label{prop:charpol_tau1}
The characteristic polynomial for the linear operator $A(\tau_1)$ induced from the quantum multiplication in $\QH^*(\OG)$ is $\lambda^{2n}-4\lambda$. 
\end{prop}
\begin{proof}
The basis of this proof is to swap the rows around, moving the top row to the bottom and shifting the rest up by one. This will allow us to get an ``almost" diagonal matrix of the form: 
\begin{tightcenter}
$\begin{blockarray}{ccccccccc}
\begin{block}{c(cccccccc)}
r_1      &1       &  -\lambda&           &        &         &         &          & 1\\
r_{2+i}  &        & \ddots   & \ddots    &        &         &         &          & \\
r_{n}    &        &          & 1         &-\lambda&         &         &          & \\
r_{n+1}  &        &          &           & 2      & -\lambda&         &          & \\
r_{n+2}  &        &          &           &        &   1     &-\lambda &          & \\   
r_{n+2+j}&        &          &           &        &         &  \ddots & \ddots   & \\   
r_{2n-1} &        &          &           &        &         &         & 1        & -\lambda\\    
r_{2n}   &-\lambda&          &           &        &         &         & 1        &  0\\
\end{block}
\end{blockarray}
$
\end{tightcenter}
We can reduce this matrix to a diagonal matrix by subtracting a factor times each row from the last row. We will use the property that if two rows (or two columns) of a matrix are interchanged, the value of the determinant changes sign.
With the appropriate row operations we can get a matrix of the form: 

\begin{tightcenter}
$\begin{blockarray}{ccccccccc}
\begin{block}{c(cccccccc)}
r_1      &1       &  -\lambda&           &            &         &         &          & 1\\
r_{2+i}  &        & \ddots   & \ddots    &            &         &         &          & \\
r_{n}    &        &          & 1         &-\lambda    &         &         &          & \\
r_{n+1}  &        &          &           & 2          & -\lambda&         &          & \\
r_{n+2}  &        &          &           &            &   1     &-\lambda &          & \\   
r_{n+2+j}&        &          &           &            &         &  \ddots & \ddots   & \\   
r_{2n-1} &        &          &           &            &         &         & 1        & -\lambda\\    
r_{2n}   &        &          &           &            &         &         &          &  \frac{4\lambda-\lambda^{2n}}{2}\\
\end{block} 
\end{blockarray}
$
\end{tightcenter}
This is an upper triangular matrix so its determinant is the product of its diagonal given by $1 \cdot 1 \cdot \ldots \cdot 1 \cdot 2 \cdot 1 \cdot \ldots \cdot 1 \cdot \frac{4\lambda - \lambda^{2n}}{2} = 4\lambda - \lambda^{2n}$. Finally since we swapped around an odd number of rows ($2n-1$ swaps were made) we need to multiply this determinant by a factor of $-1$ to get the determinant of $\tau_1 - \lambda I $ to be $\lambda^{2n} - 4\lambda$. Therefore the characteristic polynomial of $\tau_1$ is  $\lambda^{2n} - 4\lambda$.
\end{proof}
Recall that the eigenvalues of a matrix are the roots of its characteristic polynomial so we can factor the characteristic polynomial to get eigenvalues of the form \[\lambda \in \left\{\sqrt[2n-1]{4}e^{\frac{2 \pi j i}{2n-1}}: 0 \leq j \leq 2n-1\right\} \cup \{0\}.\] 

The next Corollary follows immediately from Proposition \ref{prop:charpol_tau1}.
\begin{cor} \label{lemma:util}If $\lambda$ is an eigenvalue of the linear operator $\tau_1$ induced from the quantum multiplication in $\QH^*(\OG)$ then
$\frac{\lambda^{2n-1}-2}{2} = \begin{cases} 
      -1 & \lambda = 0 \\
       1 & \text{otherwise} 
   \end{cases}$.
\end{cor}
We can use this corollary to prove the following theorem.
\begin{thm} \label{thm:eigenvector_tau1} The eigenvector of $A(\tau_1)$ associated to the eigenvalue $ \lambda_j=\sqrt[2n-1]{4}e^{\frac{2 \pi j i}{2n-1}}$, for some $1 \leq j \leq 2n-1$, or $\lambda_0=0$ is given as follows with respect to the basis $\{\tau_i\}_{0 \leq i \leq 2n-1}$.
$$\vec{v}_j=\left( \frac{\lambda_j^{2n-1}-2}{2},\frac{\lambda_j^{2n-2}}{2},\frac{\lambda_j^{2n-3}}{2}, \dots, \frac{\lambda_j^n}{2}, \lambda_j^{n-1}, \lambda_j^{n-2}, \dots, \lambda_j, 1 \right).$$
Finally, we have that $A(\tau_1)=PDP^{-1}$ where $P=[ \vec{v}_0^t \hspace{5pt} \vec{v}_1^t \cdots \vec{v}_{2n-1}^t]$ and $D=\mbox{diag}(\lambda_0,\lambda_1, \cdots, \lambda_{2n-1})$.

\end{thm}
\begin{proof}
First observe that we have the following equality where the vectors on the right are written with respect to the basis $\{ \tau_i \}_{0 \leq i \leq 2n-1}$.
\begin{eqnarray*}
\tau_0+\tau_{2n-1}+\frac{1}{2}\sum_{k=1}^{n-1} \lambda_j^k \tau_{2n-1-k}+\sum_{k=n}^{2n-2} \lambda_j^k \tau_{2n-1-k}&=& \left(1, \frac{\lambda_j^{2n-2}}{2},\frac{\lambda_j^{2n-3}}{2}, \dots, \frac{\lambda_j^n}{2}, \lambda_j^{n-1}, \lambda_j^{n-2}, \dots, \lambda_j, 1 \right);\\
\tau_{2n-1}-\tau_0 &=& \left(-1, 0,0, \ldots, 0,0, 1 \right).
\end{eqnarray*}
Using Corollary \ref{lemma:util} we can write  
$$\vec{v}_j=\left( \frac{\lambda_j^{2n-1}-2}{2},\frac{\lambda_j^{2n-2}}{2},\frac{\lambda_j^{2n-3}}{2}, \dots, \frac{\lambda_j^n}{2}, \lambda_j^{n-1}, \lambda_j^{n-2}, \dots, \lambda_j, 1 \right).$$
To prove that our vector $\vec{v}_j^t$ is an eigenvector of $A(\tau_1)$ with corresponding eigenvalue $\lambda_j$ we need to multiply $A(\tau_1)$ by $\vec{v}_j^t$. To multiply the matrix by the vector we can take the dot product of each row of $A(\tau_1)$ with the vector $\vec{v}_j^t$. We index the rows of $A(\tau_1)$ by $r_i$ using the same conventions as those in Proposition \ref{prop:charpol_tau1}.

A direct calculation show that $r_1 \cdot \vec{v}_j^t=\lambda_j \frac{\lambda_j^{2n-1}-2}{2}$ and $r_i \cdot \vec{v}_j^t=\lambda_j \cdot \frac{\lambda_j^{2n-i}}{2}$ for $2 \leq i \leq 2n$. Thus, $A(\tau_1) \vec{v}_j^t=\lambda_j\vec{v}_j^t.$ The result follows.
\end{proof}


  \section{Spectural properties of $A( \tau_p)$}
  
 We will study the Spectral properties of $A( \tau_p)$ in this section. Recall that
 
 \[ A(\tau_p)=\begin{cases} 
      A(\tau_1)^p & 1\leq p \leq n-1 \\
      \frac{1}{2}A(\tau_1)^p & n\leq p\leq 2n-2\\
      \frac{1}{2}A(\tau_1)^{2n-1}-A(\tau_0)& p=2n-1.
   \end{cases}
\]

\begin{lemma} \label{eig:lemma} We have the following results.
\begin{enumerate}
\item The set of matrices $A(\tau_p)$, with $0 \leq p \leq 2n-1$, are simultaneously diagonalizable. 
\item For $1 \leq p \leq n-1$, the eigenvalues of $A(\tau_p)$ are $\displaystyle 0,\lambda^p_1, \lambda^p_2,\cdots, \lambda^p_{\frac{2n-1}{d}}$ where each nonzero eigenvalue has multiplicity $d$. 
\item For $n \leq p \leq 2n-2$, the eigenvalues of $A(\tau_p)$ are $\displaystyle 0,\frac{1}{2}\lambda^p_1, \frac{1}{2}\lambda^p_2,\cdots, \frac{1}{2}\lambda^p_{\frac{2n-1}{d}}$ where each nonzero eigenvalue has multiplicity $d$. 
\item The eigenvalues of $A(\tau_{2n-1})$ are 1, with multiplicity $2n-1$, and $-1$, with multiplicity 1.
\end{enumerate}
\end{lemma}

\begin{proof}
First observe that that $A(\tau_0)=I$ is the identity matrix. The result follows from the following obvious facts. 
\begin{enumerate}
    \item $A(\tau_p)=A(\tau_1)^p=PD^pP^{-1}$ for $1 \leq p \leq n-1$;
    \item $A(\tau_p)=\frac{1}{2}A(\tau_1)^p=\frac{1}{2}PD^pP^{-1}$ for $n\leq p \leq 2n-2$;
    \item and for $p=2n-1$ we have that \[A(\tau_{2n-1})=\frac{1}{2}PD^{2n-1}P^{-1}-PIP^{-1}=P\left(\frac{1}{2}D^{2n-1}-I\right)P^{-1}.\]
\end{enumerate} \end{proof}



 
   
 We will next present the proof of Theorem \ref{main:thm}.
\begin{proof}
Clearly by Lemma \ref{eig:lemma}, we know that the charactistic polynomial of $A(\tau_{2n-1})$ is $(\lambda-1)^{2n-1}(\lambda+1)$.

By Lemma \ref{eig:lemma}, we have that the eigenvalues of $A(\tau_p)$, for $p<n$, are
$ 4^{\frac{p}{2n-1}}  \cdot \left(e^{2\pi i\frac{ p}{2n-1}}\right)^j$ (for $0 \leq j \leq 2n-1$) with each having multiplicity $\mbox{GCD}(p,2n-1)$. Similarly, for $n \leq p <2n-1$, the eigenvalues of $A(\tau_p)$ are $4^{\frac{p}{2n-1}-1}  \cdot \left(e^{2\pi i\frac{ p}{2n-1}}\right)^j$ (for $0 \leq j \leq 2n-1$) with each having multiplicity $\mbox{GCD}(p,2n-1)$.

The characteristic polynomial of $A(\tau_p)$, for $p<n$, is
\begin{eqnarray*}
\lambda \prod_{k=0}^{d-1}\prod_{j=0}^{\frac{2n-1}{d}-1}\left(\lambda - 4^{\frac{p}{2n-1}}\left(e^{\frac{2\pi i j}{2n-1}}\right) \right)
&=& \lambda \prod_{k=0}^{d-1}\left(\lambda^{\frac{2n-1}{d}}-4^{\frac{p}{2n-1}\cdot \frac{2n-1}{d}}\right)\\
&=&\lambda \left(\lambda^{\frac{2n-1}{d}}-4^{\frac{p}{d}}\right)^{d}\\
&=&\lambda \left(\lambda^{\frac{2n-1}{d}}-2^{\frac{2p}{d}}\right)^{d}
\end{eqnarray*}

The characteristic polynomial of $A(\tau_p)$, for $n \leq p<2n-1$, is
\begin{eqnarray*}
&&\lambda \prod_{k=0}^{d-1}\prod_{j=0}^{\frac{2n-1}{d}-1}\left(\lambda - \frac{1}{2}\cdot 4^{\frac{p}{2n-1}}\left(e^{\frac{2\pi i j}{2n-1}}\right)\right)\\
&=& \lambda \prod_{k=0}^{d-1}\left(\lambda^{\frac{2n-1}{d}}-\left(\frac{1}{2}\cdot 4^{\frac{p}{2n-1}}\right)^{\frac{2n-1}{d}}\right)\\
&=&\lambda \prod_{k=0}^{d-1}\left(\lambda^{\frac{2n-1}{d}}-\left(\frac{1}{2}\right)^{\frac{2n-1}{d}}\cdot \left(4^{\frac{p}{d}}\right)\right)\\
&=&\lambda \prod_{k=0}^{d-1}\left(\lambda^{\frac{2n-1}{d}}-\left(4^{-\frac{2n-1}{2d}}\right)\cdot \left(4^{\frac{p}{d}}\right)\right)\\
&=&\lambda \prod_{k=0}^{d-1}\left(\lambda^{\frac{2n-1}{d}}-\left(2^{-\frac{2n-1}{d}}\right)\cdot \left(2^{\frac{2p}{d}}\right)\right)\\
&=&\lambda \prod_{k=0}^{d-1}\left(\lambda^{\frac{2n-1}{d}}-\left(2^{\frac{2p-(2n-1)}{d}}\right)\right)\\
&=&\lambda \left(\lambda^{\frac{2n-1}{d}}-2^{\frac{2p-(2n-1)}{d}}\right)^{d}
\end{eqnarray*}
The result follows.\end{proof}

 \section{Galkin's lower bound conjecture for $\OG$}
 
Lastly we check that Galkin's Lower Bound Conjecture holds for $H^\bullet(\OG)$.
\begin{thm} \label{prop:galkin}
Galkin's Lower Bound Conjecture holds for $H^\bullet(\OG)$. That is, \[ \text{FPdim} ((2n-1)\tau_1) \geq \dim \OG +1.\]
\end{thm}
\begin{proof}
The smooth real function $f(x):=4^x-x-1$ satisfies $f(0)=0$ and $f'(x)=4^x \ln(4)-1>0$. Hence, $f\left(\frac{1}{2n-1}\right)>0$. That is, $4^{\frac{1}{2n-1}}>\frac{1}{2n-1}+1=\frac{2n}{2n-1}$. The result follows.
\end{proof}
 
\bibliography{Oconj}
\bibliographystyle{halpha}
\end{document}